\documentclass{amsart}
\usepackage{amsthm,amssymb}

\usepackage{hyperref} 

\usepackage{xcolor} 

\newcommand{\N}{\ensuremath{\mathbb{N}}}
\newtheorem{defi}{Definition}
\newtheorem{theorem}[defi]{Theorem}
\newtheorem{lemma}[defi]{Lemma}
\newtheorem{cor}[defi]{Corollary}
\newtheorem{ques}[defi]{Question}

\title{The Vietoris hyperspace of finite sets of Erd\H{o}s space}
\author[A. Zaragoza]{Alfredo Zaragoza}
\address[A. Zaragoza]{Departamento de Matemáticas, Facultad de Ciencias, Universidad Nacional Autónoma de México, Circuito Exterior s/n, Ciudad Universitaria, Coyoacán, 04510, Mexico city, Mexico
}
\email[A. Zaragoza]{soad151192@icloud.com}
\thanks{This work is part of the doctoral work of the author at UNAM, Mexico city, under the direction of R. Hernández-Gutiérrez. The author was supported by a CONACyT doctoral scholarship with number 696239}

\begin{document}

\begin{abstract}
Recently, David S. Lipham has shown that if $X$ is an Erd\H{o}s space factor then the Vietoris hyperspace $\mathcal{F}(X)$ of finite subsets of $X$ is an Erd\H{o}s space factor. In this short note we prove that if $\mathfrak{E}$ denotes Erd\H{o}s space then $\mathcal{F}(\mathfrak{E})$ is in fact homeomorphic to $\mathfrak{E}$.
\end{abstract}

\maketitle

\section{Introduction}
Every topological space in this article is assumed to be metrizable and separable.
Erd\H{o}s space is defined to be the space
$$ \mathfrak{E} = \{(x_n)_{n\in \omega} \in  \ell^2 : x_i \in  \mathbb{Q} , \textit{for all } i\in \omega \},$$
where $\ell^2$ is the Hilbert space of square-summable sequences of real numbers.
Erd\H{o}s space was introduced by Erd\H{o}s in 1940 in \cite{er} as an example of a totally disconnected and non-zero dimensional space.
For a space $X$, $\mathcal{K}(X)$ denotes the hyperspace of non-empty compact subsets of $X$ with the
Vietoris topology. For any $n\in  \N$, $\mathcal{F}_n(X)$ 
is the subspace of $\mathcal{K}(X)$ consisting
of all non-empty subsets that have cardinality less than or equal to $n$, and $\mathcal{F}(X)$ is the subspace of $\mathcal{K}(X)$ of finite subsets of $X$.

In \cite{zaragoza}, the author proved that $\mathcal{F}_n(\mathfrak{E})$ is homeomorphic to $\mathfrak{E}$ for all $n\in\mathbb{N}$. Soon after that, David S. Lipham proved in \cite{lipham} that if $X$ is an Erd\H{o}s space factor then $\mathcal{F}(X)$ is an Erd\H{o}s space factor. The objective of this note is to round off this topic with the following result. 

\begin{theorem}\label{FE}
$\mathcal{F}(\mathfrak{E})$ is homeomorphic to $\mathfrak{E}$.
\end{theorem}

We will also give some indirect consequences of the theorem.

\section{Definitions and preliminaries}

In this paper, $\N$ is the set of positive integers and $\omega=\{0\}\cup\N$ is the set of natural numbers. Given $n\in \N$ and subsets $U_1,\ldots, U_n$ of a topological space $X$, $\langle  U_{1},\ldots ,U_{n}\rangle$ denotes the collection $\left\lbrace   F \in \mathcal{K}(X):F\subset \bigcup_{k=1}^n U_k,\, F\cap U_{k}\neq \emptyset \textit{ for } k \leq n \right\rbrace $. Recall that the Vietoris topology on $\mathcal{K}(X)$ has as its canonical base all the sets of the form $\langle  U_{1},\ldots ,U_{n}\rangle$, where $U_k$ is a non-empty open subset of $X$ for each $k\leq n$.

In \cite{ME}, Dijkstra and van Mill gave an intrinsic characterization of Erd\H{o}s space. We will use this characterization so we first recall some definitions.

\begin{defi}[{\cite[Remark 2.4]{ME}}]\label{eazd} 
A topological space $(X,\mathcal{T})$ is almost zero-dimensional if
there is a zero-dimensional topology $\mathcal{W}$ in $X$ such that $\mathcal{W}$ is coarser than $\mathcal{T}$ and has the property
that every point in $X$ has a local neighborhood base consisting of sets that are closed with respect
to $\mathcal{W}$.
\end{defi}

\begin{defi}[{\cite[Definition 5.1]{ME}}] 
Let $X$ be a space and let $\mathcal{A}$ be a collection of subsets of $X$.
The space $X$ is called $\mathcal{A}$-cohesive if every point of the space has a neighborhood
that does not contain nonempty proper clopen subsets of any element of $\mathcal{A}$. 
\end{defi}

We refer the reader to \cite{Ke} for the basic theory of trees.

\begin{theorem}[{\cite[Theorem 8.13, p. 46]{ME}}]\label{EE} A nonempty space $E$ is homeomorphic to $\mathfrak{E}$ if and only if there exists topology $\mathcal{W}$ on $E$ that witnesses the almost zero-dimensionality of $E$, a nonempty tree $T$ over a countable alphabet and subspaces $E_s$ of $E$ that are closed with
respect to $\mathcal{W}$ for each $s\in T$ such that:

\begin{enumerate}
\item $E_{\emptyset} = E$ and $E_s =\bigcup \{E_t : t \in succ(s)\}$ whenever $s \in T$,
\item each $x\in E$ has an open neighborhood $U$ that is an anchor in $(E,\mathcal{W})$, that is, for every $t\in [T ]$ we either have $E_{t\restriction k} \cap U = \emptyset$ for some
$k \in \omega$ or the sequence $E_{t\restriction k_0},\ldots, E_{t\restriction n},\ldots $ converges in $(E,\mathcal{W})$ to a point.

\item for each $s \in T$ and $t \in succ(s)$,  $E_t$ is nowhere dense in $E_s$,
and
\item  $E$ is $\{Es : s \in T\}$-cohesive.
\end{enumerate}
\end{theorem}

For each $n\in\N$, let $\varphi_n:X^n\to \mathcal{F}_n(X)$ be the function defined by $\varphi_n(x_1,\ldots,x_n)=\{x_1,\ldots x_n\}$. It is know that this function is continuous, finite-to-one and in fact it is a quotient \cite{Sub}. 

\begin{lemma}\label{Coh}
Let $X$ be a space that is $\{A_s: s\in S\}$-cohesive, witnessed by a base $\mathcal{B}$ of open sets. Consider the following collection of subsets of $\mathcal{F}(X)$:
$$
\mathcal{A}=\left\{\varphi_n[A_{s_1}\times\cdots\times A_{s_n}]:n\in\mathbb{N},\, \forall i\in\{1,\ldots,n\}\, (s_i\in S)\}\right.
$$
Then $\mathcal{F}(X)$ is $\mathcal{A}$-cohesive, and the open sets that witness this may be taken from the collection $\mathcal{C}=\left\{\langle U_1,\ldots, U_n\rangle:\forall i\in\{1,\ldots,n\}\, (U_i\in\mathcal{B})\right\}$.
\end{lemma}
\begin{proof}
Let $F\in \mathcal{F}(X)$, suppose that $F=\{x_1, \ldots ,x_k\}$ with $x_j\neq x_i$ if $i\neq j$. For each $j\in\{1,\ldots,n\}$, let $V_j\in \mathcal{B}$ with $x_j\in V_j$. We can assume that if $i\neq j$ then $V_i\cap V_j=\emptyset$. Let $V=\langle V_1,\ldots, V_k\rangle$, note that $F\in V$. We claim that $V$ does not contain clopen subsets of any element of $\mathcal{A}$.

Suppose there are $s_1,\ldots s_m \in S$ such that $V$  contains a non-empty proper clopen subset $O$ of $\varphi_m[A_{s_1}\times\cdots\times A_{s_m}]$. As $V\cap\mathcal{F}_{k-1}(X)=\emptyset$, it follows that $m\geq k$.  If $i\in(k,m]$, 
we define $V_i=V_k$. In this way, $V=\langle V_1,\ldots, V_m\rangle$.
Thus $O\cap \mathcal{F}_m(X)$ is a clopen subset of $ \mathcal{F}_m(X)$. Let $x=(x_1, \ldots,x_{k}, x_{k+1},\ldots, x_m)$ where $x_k=x_{k+1}=\ldots =x_m$. Note that $x\in V_1\times \ldots \times V_m $.
Let $g_m=\varphi_m\restriction A_{s_{1}}\times \ldots \times A_{s_{{m}}}$ and $C=g_m^{\leftarrow}[O\cap \mathcal{F}_m(X)]\cap (V_1\times \ldots \times V_m)$. By Proposition  2.6 of \cite {zaragoza}, $C$ is a clopen subset of $A_{s_{1}}\times \ldots \times A_{s_{{m}}}$ such that $C\subset V_1\times \ldots \times V_m $; this is a contradiction ( see \cite[Remark 5.2]{van}).
\end{proof}

\section{Proof of the theorem}
By Theorem \ref{EE} there is a topology $\mathcal{W}$ for $\mathfrak{E}$ which is witness to the almost zero dimensionality of $\mathfrak{E}$, a countable tree $ T $, a family of sets $\mathcal{E}=\{ E_s :s\in T\}$ 
which are closed with respect to $\mathcal{W}$, and for each $x\in\mathfrak{E}$ an open neighborhood $U_x$ that satisfy the conditions of Theorem \ref{EE} for $\mathfrak{E}$.

Let $\mathcal{W}_n$ the topology of $\mathcal{F}_n((\mathfrak{E},\mathcal{W}))$, $T^n=\{s_1*\ldots *s_n: s_1, \ldots, s_n\in T \textit{ and } \vert s_1\vert =\ldots =\vert s_n\vert\}$ and for each $s_1*\ldots *s_n\in T^n$ let $H_{s_1*\ldots* s_n}$ be the subset $\varphi_n[E_{s_1}\times\cdots\times E_{s_n}]$ of $\mathcal{F}_n(\mathfrak{E})$. 

We need to define the neighborhoods that will work as anchors. However, in Theorem 1.2 of \cite{zaragoza}, the neighborhoods that we proposed as anchors depend on the branch. In the following Lemma we correct that mistake.

Let $F=\{x_1,\ldots,x_k\}\in\mathcal{F}(\mathfrak{E})$. For each $j\leq k$, 
let $U_{x_j} $ be a neighborhood of $ x_j $ which is anchor in $(\mathfrak{E},\mathcal{W})$. Let $\mathcal{U}_F=\langle U_{x_1},\ldots, U_{x_k}\rangle$.

\begin{lemma}\label{ancla}
If $F\in\mathcal{F}_n(X)$, 
the set $\mathcal{U}_F\cap \mathcal{F}_n(X)$ is an anchor for $(\mathcal{F}_n(\mathfrak{E}),\mathcal{W}_n)$.
\end{lemma}
\begin{proof}

Let $ F=\{x_1,\ldots,x_k\}\in\mathcal{F}_n(\mathfrak{E})$, and  $\mathcal{U}_F$ defined as above. Consider $ t \in [T^n] $. 
If there exist $i\in \omega$ such that $\mathcal{U}_F\cap\varphi_n[E_{t_1\restriction i}\times \ldots \times E_{t_n\restriction i}]=\emptyset$ then we are finished. Now suppose that for each $ i \in \omega $ 
$$ (*)\textit{  }\mathcal{U}_F \cap \varphi_n[E_{t_1\restriction i} \times \ldots \times E_{t_n\restriction i}] \neq\emptyset.$$
We claim that for all $ i \in \omega $ and $ j \leq n $, there exists $l(i,j)\in \omega$ such that, $U_{l(i,j)}\cap E_{t_j\restriction i} \neq \emptyset$.
By $(*)$ there exists $(y_1,\ldots, y_n)\in E_{t_1\restriction i}, \times \ldots \times E_{t_n\restriction i}$ such that $$\{y_1,\ldots, y_n\}=\varphi_n(y_1,\ldots, y_n)\in \mathcal{U}_F=\langle U_{x_1},\ldots U_{x_k}\rangle.$$
So there is $l\leq n$ such that $y_j\in U_{x_l}$; this proves the claim.

If we fix $j$, this defines a function $i \mapsto l(i,j) $ with domain $\omega $ and codomain $\{1,\ldots, k\}$. This implies that there exists an infinite $ A \subset \omega$  and fixed $l_j\in \N $ such that $E_{t_j\restriction i}\cap U_{l_j}\neq\emptyset$ for each $i\in A$.
As $E_{t_j\restriction s} \subset E_{t_j\restriction t}$ if $s<t$, then $E_{t_j\restriction i}\cap U_{l_j}\neq\emptyset$ for all $i\in \omega$.
Therefore $\{E_{t_j\restriction i}:i\in \omega\}$ converges in $(\mathfrak{E},\mathcal{W})$ to a point $p_j\in \mathfrak{E}$; this holds for all $j\in\{1,\ldots,n\}$.

Therefore $\{E_{t_1\restriction i} \times \ldots\times E_{t_n\restriction i}: i \in \omega\}$ converges in $(\mathfrak{E}^n,\mathcal{W}^n)$ to $(p_1,\ldots,p_n)$. So $\{\varphi_n[E_{t_1\restriction i} \times \ldots\times E_{t_n\restriction i}]: i \in \omega\}$ converges in $\mathcal{F}_n(\mathfrak{E},\mathcal{W})$ to $\varphi_n((p_1,\ldots,p_n))$.

\end{proof}

We are ready to prove the theorem.

\begin{proof}[Proof of Theorem \ref{FE}]

In Theorem 1.2 from \cite{zaragoza} it is proved that for each $n\in \mathbb{N}$ the topology $\mathcal{W}_n$, the tree $T^n$ and collection $\mathcal{E}_n=\{\varphi_n(E _{s_1}\times \ldots \times  E_{s_n }): s_1*\ldots *s_n\in T^n \}$ satisfy conditions (1), (3) and (4) of Theorem \ref{EE} for $\mathcal{F}_n(\mathfrak{E})$. Also, by Lemma \ref{ancla}, $\{\mathcal{U}_F:F\in\mathcal{F}_n(X)\}$ are anchors so condition (2) of Theorem \ref{EE} for $\mathcal{F}_n(\mathfrak{E})$ is also satisfied. 
Consider the tree
$$T=\{\emptyset\}\cup \{\langle n\rangle^{\frown} s_1*\ldots *s_n :n\in\omega,\, s_1*\ldots *s_n\in T^n\}.$$ 
Let $X_{\emptyset}=\mathcal{F}(\mathfrak{E})$ and $X_{\langle n\rangle^{\frown} s_1*\ldots *s_n }= \varphi_n[E _{s_1}\times \ldots \times  E_{s_n }]$ for each $n\in \omega$ and $s_1*\ldots *s_n \in T^n$. Let $\mathcal{W}^\prime$ be the Vietoris topology in $\mathcal{F}(\mathfrak{E}, \mathcal{W})$. 
 Then we will prove that $\mathcal{W}^\prime$, $T$, $\mathcal{S}=\{X_s: s\in T\}$ and $\{\mathcal{U}_F: F\in \mathcal{F}(\mathfrak{E})\}$ satisfy the conditions required in Theorem \ref{EE} for $\mathcal{F}(\mathfrak{E})$. Indeed, the fact that the Vietoris topology in $\mathcal{F}(\mathfrak{E}, \mathcal{W})$ witnesses that $\mathcal{F}(\mathfrak{E})$ is almost zero-dimensional follows from the proof of Proposition 2.2 in \cite{zaragoza}.
On the other hand, for each $s_1*\ldots * s_n \in T_n$, $\varphi_n[E_{s_1}\times \ldots \times  E_{s_n}]$ is a closed
subset of $\mathcal{F}_n(\mathfrak{E}, \mathcal{W})$, hence  $\varphi_n[E_{s_1}\times \ldots\times E_{s_n}]$ is closed in $\mathcal{F}(\mathfrak{E}, \mathcal{W})$, because  $\mathcal{F}_n(\mathfrak{E}, \mathcal{W})$ is closed in $\mathcal{F}(\mathfrak{E}, \mathcal{W})$. 

For $\emptyset\in T$, we have that $succ_T(\emptyset)= \{\langle n \rangle^\frown \emptyset_1*\ldots * \emptyset_n: n\in \omega, \emptyset_1*\ldots * \emptyset_n\in T_n \}$ then $X_{\langle n \rangle^\frown \emptyset_1*\ldots * \emptyset_n}=\mathcal{F}_n(\mathfrak{E})$. Hence $X_{\langle n \rangle^\frown \emptyset_1*\ldots * \emptyset_n}$ is nowhere dense in $X_{\emptyset}$ and $X_{\emptyset}=\bigcup_{t\in succ_T(\emptyset)} X_{t}$. On the other hand, if $\langle n\rangle^{\frown} s_1*\ldots *s_n\in T\setminus \{\emptyset\}$, $succ_T(\langle n\rangle^{\frown} s_1*\ldots *s_n )=\{ \langle n\rangle^{\frown} t: t\in succ_{T^n}(s_1*\ldots *s_n )\}$.
 Then for each $\langle n\rangle^{\frown} s_1*\ldots *s_n \in T\setminus \{\emptyset\}$, we have $X_{\langle n\rangle^{\frown} s_1*\ldots *s_n }=\bigcup\{X_{\langle n\rangle^{\frown} t}: t\in succ_{T^n}(s_1*\ldots *s_n)\}$ and $X_{\langle n\rangle^{\frown} t}$ is nowhere dense in $X_{\langle n\rangle^{\frown} s_1*\ldots *s_n}$ if $t\in succ_T(s_1*\ldots *s_n)$. Thus, conditions (1) and (3) of Theorem \ref{EE} for $\mathcal{F}(\mathfrak{E})$ are satisfied. Now we prove condition (2). Let $F\in\mathcal{F}(\mathfrak{E})$. Suppose that $F=\{x_1,\ldots, x_k\}$ with $x_j\neq x_i$ if $i\neq j$ so $\mathcal{U}_F=\langle U_{x_1},\ldots, U_{x_k}\rangle$. We claim that $\mathcal{U}_F$ is an anchor in $\mathcal{W}^\prime$. Let $\widehat{t}\in[T]$, then there exists $n\in \omega$ such that $\widehat{t}=\langle n\rangle^\frown \widehat{t}_n$ and $\widehat{t}_n\in [T^n]$. Also, there exist $t_i\in [T]$ such that $\widehat{t}_n=(t_1,\ldots, t_n)$.\\\\
 \noindent\textbf{Case 1}:
  $k\leq n$. By Lemma \ref{ancla}, $\mathcal{U}_F\cap\mathcal{F}_n(\mathfrak{E})$ is a anchor of $F$ in $(\mathcal{F}_n(\mathfrak{E}),\mathcal{W}_n)$. 
We will show that $\mathcal{U}_F\cap X_{\widehat{t}\restriction i}=\emptyset$ for some $i\in \omega$, or that $(X_{\widehat{t}\restriction i})_{i\in \omega}$ converges in $(\mathcal{F}(\mathfrak{E}),\mathcal{W}^\prime)$. If $\mathcal{U}_F\cap X_{\widehat{t}\restriction i}=\emptyset$ for some $i\in \omega$, we are finished. If $\mathcal{U}_F\cap X_{\widehat{t}\restriction i}\neq\emptyset$ for each $i\in \omega$, as $X_{\widehat{t}\restriction {i+1}}=X_{\langle n\rangle^\frown \widehat{t}_n \restriction i} =\varphi_n[E_{t_1\restriction i}\times \ldots \times E_{t_n\restriction i}]\subset \mathcal{F}_n(\mathfrak{E})$ then $(\mathcal{U}_F\cap \mathcal{F}_n(\mathfrak{E}))\cap X_{\widehat{t}\restriction i}\neq\emptyset$ for each $i\in \omega$. Hence $(X_{\widehat{t}\restriction i})_{i\in \omega}$ converges in $(\mathcal{F}(\mathfrak{E}),\mathcal{W}^\prime)$ because $\mathcal{U}_F\cap\mathcal{F}_n(\mathfrak{E})$ is a anchor of $F$ in $(\mathcal{F}_n(\mathfrak{E}),\mathcal{W}_n)$.
So $\mathcal{U}_F$ is anchor of $F$ in $(\mathcal{F}(\mathfrak{E}),\mathcal{W}^\prime)$.
\\\\
\noindent \textbf{Case 2}:
  $k> n$. For each $j\in (n,k]$, 
let's consider $t_{n+1}=\ldots =t_k=t_n$. Let $\widehat{s}=(t_1,\ldots, t_n, t_{n+1}, \ldots, t_k)$, then $\widehat{s}\in [T^k]$. From case 1, we have that $X_{\langle k\rangle^\frown\widehat{s}\restriction j} \cap \mathcal{U}_F=\emptyset$ if $j\geq m$ or $(X_{\langle k\rangle^\frown\widehat{s}\restriction j} )_{j\in \omega}$ converges. We claim that either $X_{\widehat{t}\restriction j} \cap \mathcal{U}_F=\emptyset$ for some $j\geq m$ or $(X_{\widehat{t}\restriction j} )_{j\in \omega}$ converges. 
Note that for all $j\in\omega$ $$(*)\textit{   }\varphi_n[ E_{t_1\restriction j}\times \cdots\times E_{t_n\restriction j}]\subset \varphi_k[ E_{t_1\restriction j}\times \ldots\times E_{t_n\restriction j}\times E_{t_{n+1}\restriction j}\times\cdots\times E_{t_k\restriction j}].$$
Let's suppose that $X_{\widehat{t}\restriction j} \cap \mathcal{U}_F\neq\emptyset$ for each $j\in \omega$. By $(*)$ we have $X_{\langle k\rangle^\frown\widehat{s}\restriction j} \cap \mathcal{U}_F\neq\emptyset$, then $(X_{\langle k\rangle^\frown\widehat{s}\restriction j})_{j\in \omega}$ converges to some $A\in\mathcal{F}_k(\mathfrak{E})$.
Hence $(\varphi_n[ E_{t_1\restriction j}\times \cdots\times E_{t_n\restriction j}])_{j\in \omega}$ converges to $A$. Then $(X_{\widehat{t}\restriction j} )_{j\in \omega}$ converges to $A$.\\\\
Thus, condition (2) of Theorem \ref{EE} for $\mathcal{F}(\mathfrak{E})$ is satisfied. By Lemma \ref{Coh}, $\mathcal{F}(\mathfrak{E})$ is $\{X_{\langle n\rangle^{\frown} s} : n\in \N, s\in T^n\}$-cohesive. Let $F\in \mathcal{F}(\mathfrak{E})$ such that $F\in \mathcal{W}$ and let $\mathcal{W}\in \mathcal{C}$ be the open set given by Lemma \ref{Coh} that does not contain non-empty clopen subsets of any element of $\{X_{\langle n\rangle^{\frown} s} : n\in \N, s\in T^n\}$.
We claim that $\mathcal{W}$ does not contain non-empty clopen subsets of $\mathcal{F}(\mathfrak{E})$. Suppose otherwise, then there is an clopen non-empty $\mathcal{O}$ subset of $\mathcal{F}(\mathfrak{E})$ such that $O\subset \mathcal{W}$. Note that  for each $n\in \mathbb{N}$, we have $\mathcal{O}\cap X_{\langle n\rangle^{\frown} s}=\emptyset$ because $\mathcal{W}$ does not contain clopen of any element of  $\{X_{\langle n\rangle^{\frown} s} : n\in \N, s\in T^n\}$. Thus, $\mathcal{O}$ is empty and this is a contradiction. This implies that in fact $\mathcal{F}(\mathfrak{E})$ is $\{X_t:t\in T\}$-cohesive. Thus, condition (4) of Theorem \ref{EE} for $\mathcal{F}(\mathfrak{E})$ is satisfied. Then by Theorem \ref{EE} we have  that $\mathcal{F}(\mathfrak{E})$ is homeomorphic to $\mathfrak{E}$.
\end{proof}
\section{An application to Sierpiński stratifications }

A system of sets $(X_s)_{s\in T}$ is a Sierpiński stratification (\cite[Definition 7.1, p. 31]{ME}) of a space $X$ if:
\begin{enumerate}
\item T is a non-empty tree over a countable alphabet,
\item each $X_s$ is a closed subset of $X$,
\item $X_{\emptyset}=X$, $X_s=\bigcup\{X_t:t\in succ(s)\}$
\item if $\sigma\in [T]$ then the sequence $X_{\sigma\restriction 1},\ldots, X_{\sigma\restriction n},\ldots$ converges to a point in $X$.
\end{enumerate}

Note that in Theorem 1.1 of \cite{zaragoza} it is implicitly proved that if $(X_s)_{s\in T}$ is a Sierpiński stratification of a space $X$, then $\mathcal{C}_n=\{\varphi_n[X_{s_1}\times \ldots\times X_{s_n}]: s_1*\ldots *s_n\in T^n \}$ is a Sierpiński stratification of $\mathcal{F}_n(X)$. Define $X_\emptyset=X$ and $ X_{\langle n\rangle^{\frown} s_1*\ldots *s_n }=\varphi_n[X_{s_1}\times \ldots\times X_{s_n}]$ for $n\in\N$ and $s_1*\ldots *s_n\in T^n$. Then by an argument similar to the proof of Theorem \ref{FE} we obtain that $\mathcal{C}=\{X_\emptyset\}\cup\{X_{\langle n\rangle^{\frown} s_1*\ldots *s_n }:n\in \N, s_1*\ldots *s_n\in T^n\}$   is a Sierpiński stratification of $\mathcal{F}(X)$. So we have the following Corollary.
\begin{cor}\label{se}
Let $(X_s)_{s\in T}$ be Sierpiński stratification of a space $X$. Then 
\begin{enumerate}
\item For each $n\in \mathbb{N}$, $\mathcal{C}_n$ is a Sierpiński stratification of $\mathcal{F}_n(X)$.
\item   $\mathcal{C}$   is a Sierpiński stratification of $\mathcal{F}(X)$. 

\end{enumerate}
\end{cor}

Van Engelen proved in (\cite[Theorem A.1.6]{van})  that a zero dimensional $ X $ space is homeomorphic to $\mathbb{Q}^\omega$ if $X$ has a Sierpiński stratification $(X_s)_{s\in T}$ such that $X_t$ is nowhere dense in $X_s$, if $t\in succ(s)$. Using this fact, the following corollary can be proved.

\begin{cor}
The spaces $\mathcal{F}_n(\mathbb{Q}^\omega)$, $\mathcal{F}(\mathbb{Q}^\omega)$ are homeomorphic  to $\mathbb{Q}^\omega$.
\end{cor}
\begin{proof}
Let $(X_s)_{s\in T}$ be Sierpiński stratification of a space $\mathbb{Q}^\omega$ such that $B_t$ is nowhere dense in $B_s$, if $t\in succ(s)$. By Corollary \ref{se} for each $n\in \mathbb{N}$, $\mathcal{C}_n$ and $\mathcal{C}$ are Sierpiński stratifications of $\mathcal{F}_n(X)$ and  $\mathcal{F}(X)$, 
respectively. Notice that we have that $\varphi_n[X_{t_1}\times \ldots\times X_{t_n}]$ is nowhere dense in  $\varphi_n[X_{s_1}\times \ldots\times X_{s_n}]$ if $t_1*\ldots* t_n\in succ(t_1*\ldots* t_n)$ and that $ X_{\langle n\rangle^{\frown} t_1*\ldots *t_n}$ is nowhere dense in $X_{\langle n\rangle^{\frown} s_1*\ldots *s_n}$ if $\langle n\rangle^{\frown} t_1*\ldots *t_n\in succ(\langle n\rangle^{\frown} s_1*\ldots *s_n)$. 
\end{proof}

David Lipham has proved in \cite{lipham} that $ X $ is an $ \mathfrak{E} $ - factor, if and only if it admits a Sierpi\'nski stratification $ (B_s) _ {s \in T} $. We can
consider the sets $ B_s $ as subsets of $ (X, \mathcal{W}) $, where $ \mathcal{W} $ is a topology witness to the almost zero dimensionality of $ X $. A natural question is this.
\begin{ques}
 If $ X $ is a $ \mathfrak {E} $ - factor, under what conditions is $ (X, \mathcal{W}) $ homeomorphic to $ \mathbb{Q}^ \omega $?
\end{ques}

The next natural question in this direction is whether the hyperspace $\mathcal{K}(\mathfrak{E})$ is homeomorphic to $\mathfrak{E}$. However, in a personal communication Professor Jan van Mill explained that $\mathcal{K}(\mathfrak{E})$ is not Borel. The argument goes as follows. By Theorem 9.2 in \cite{ME} there is a closed copy of $\mathbb{Q}$ in $\mathfrak{E}$. Thus, there is a closed copy of $\mathcal{K}(\mathbb{Q})$ inside $\mathcal{K}(\mathfrak{E})$. However, $\mathcal{K}(\mathbb{Q})$ is not Borel by a result of Hurewicz (see \cite[33.5]{Ke}). This implies that $\mathcal{K}(\mathfrak{E})$ is not Borel.

However, there is another direction that is worth exploring. Michalewski proved in \cite{mich} that $\mathcal{K}(\mathbb{Q})$ is a topological group. Also, the following follows directly from Theorem \ref{FE}.

\begin{cor}
$\mathcal{F}(\mathfrak{E})$ is homogeneous.
\end{cor}

Thus, the following is a natural question.

\begin{ques}
Is $\mathcal{K}(\mathfrak{E})$ homogeneous?
\end{ques}

\end{document}